\documentclass{amsart}

\theoremstyle{definition}

\theoremstyle{remark}

\numberwithin{equation}{section}

\newcommand{\CD}{\mathcal{D}}

\newcommand{\CI}{\mathcal{I}}

\newcommand{\CO}{\mathcal{O}}

\newcommand{\Ann}{{\rm Ann}}
\newcommand{\rint}{{\rm int}}
\newcommand{\rcl}{{\rm cl}}
\newcommand{\coz}{{\rm Coz}}

\newcommand{\Min}{{\rm Min}}
\newcommand{\Spec}{{\rm Spec}}

\newcommand{\sub}{\supseteq}

\newcommand{\la}{\lambda}
\newcommand{\La}{\Lambda}

\newcommand{\ga}{\gamma}

\newcommand{\al}{\alpha}

\newcommand{\de}{\delta}
\newcommand{\be}{\beta}
\newcommand{\si}{\sigma}

\newcommand{\cc}{\circ}
\newcommand{\tohi}{\emptyset}

\newcommand{\bcap}{\bigcap}

\newcommand{\ov}{\overline}

\newcommand{\RTA}{\Rightarrow}
\newcommand{\LTA}{\Leftarrow}

\def\ifif {if and only if\ \ }
\def\sub {\subseteq}

\def\set {\setminus }
\def\RTA {\Rightarrow }
\def\LTA {\Leftarrow }

\def\BBR{\mathbb{R}}

\theoremstyle{definition}\newtheorem{thm}{Theorem}[section]
\theoremstyle{definition}\newtheorem{cor}[thm]{Corollary}
\theoremstyle{definition}\newtheorem{lem}[thm]{Lemma}
\theoremstyle{definition}\newtheorem{prop}[thm]{Proposition}
\theoremstyle{definition}\newtheorem{defn}[thm]{Definition}
\theoremstyle{definition}\newtheorem{Rem}[thm]{Remark}
\theoremstyle{remark}
\theoremstyle{definition}\newtheorem{exam}[thm]{Example}
\theoremstyle{definition}

\begin{document}
\title[On $WCF$-spaces]
{An extension of $F$-spaces and its applications}


\author{A.R. Aliabad}
\address{Department of Mathematics, Sahid Chamran University, Ahvaz, Iran\newline \indent}
\email{aliabadyr@scu.ac.ir}

\author{A. Taherifar}
\address{Department of Mathematics, Yasouj University, Yasouj, Iran}
\email{ataherifar@yu.ir, ataherifar54@gmail.com}



\begin{abstract}
A completely regular Hausdorff space $X$ is called a $WCF$-space if every pair of disjoint cozero-sets in $X$ can be separated by two disjoint $Z^{\circ}$-sets. The class of $WCF$-spaces properly contains both the class of $F$-spaces and the class of cozero-complemented spaces. We prove that if $Y$ is a dense $z$-embedded subset of a space $X$, then $Y$ is a $WCF$-space if and only if $X$ is a $WCF$-space. As a consequence, a completely regular Hausdorff space $X$ is a $WCF$-space if and only if $\beta X$ is a $WCF$-space if and only if $\upsilon X$ is a $WCF$-space. 
We then apply this concept to introduce the notions of $PW$-rings and $UPW$-rings. A ring $R$ is called a $PW$-ring (resp., $UPW$-ring) if for all $a, b \in R$ with $aR \cap bR = 0$, the ideal $\Ann(a)+\Ann(b)$ contains a regular element (resp., a unit element). It is shown that $C(X)$ is a $PW$-ring if and only if $X$ is a $WCF$-space, if and only if $C^{*}(X)$ is a $PW$-ring. Moreover, for a reduced $f$-ring  $R$ with bounded inversion, we prove that the lattice $BZ^{\circ}(R)$ is co-normal if and only if $R$ is a $PW$-ring. Several examples are provided to illustrate and delimit our results. 
\end{abstract}
\maketitle
\section{Introduction}
In this paper, all topological spaces are assumed to be completely regular Hausdorff, and all rings are commutative with unity.  
It is well known that the collection of all cozero-sets in a completely regular Hausdorff space $X$ forms a base for the open sets.  
This highlights the fundamental role of cozero-sets in the characterization of such spaces.  
Moreover, cozero-sets have been used for introducing and studding of several important classes of spaces, such as $F$-spaces, $F^{\prime}$-spaces, and cozero-complemented spaces (see \cite{GJ, GM, GME, NH, HW, HVW}).  
 In addition, the notion of a $WED$-space were introduced in \cite{AT} and \cite{DT},   
in which, every pair of disjoint open sets in $X$ can be separated by two disjoint $Z^{\circ}$-sets.  

Motivated by these considerations, we introduce a new class of spaces, called $WCF$-spaces. In the definition of a $WED$-space, open sets are replaced by cozero-sets, which leads to a broader class of topological spaces. We show that the class of $WCF$-spaces properly contains the classes of $F$-spaces, cozero-complemented spaces, and $WED$-spaces. In Section~2, we recall the necessary background and fix the notation to be used throughout the paper.

In Section~3, we investigate several topological properties of $WCF$-spaces. Examples are provided to illustrate the significance of the subject. It is proved that if $Y$ is a dense and $Z$-embedded subset of a topological space $X$, then $X$ is a $WCF$-space if and only if $Y$ is a $WCF$-space (Theorem~\ref{sara}). As a consequence, every dense and $C^{*}$-embedded subset of a $WCF$-space is itself a $WCF$-space. Hence, we deduce that $X$ is a $WCF$-space if and only if $\beta X$ is a $WCF$-space if and only if $\upsilon X$ is a $WCF$-space.  

In Section~4, we address the question: ``What is $C(X)$ when $X$ is a $WCF$-space?'' This leads us to introduce new classes of commutative rings. A ring $R$ is called a $PW$-ring (resp., $UPW$-ring) if, for each $a,b\in R$ with $aR\cap bR=0$, the ideal $\Ann(a)+\Ann(b)$ contains a regular element (resp., a unit element). We show that if $\{R_{\alpha} : \alpha \in S\}$ is a family of rings, then the product ring $R=\prod_{\alpha\in S} R_{\alpha}$ is a $PW$-ring if and only if each $R_{\alpha}$ is a $PW$-ring (Proposition~\ref{has}). For a reduced ring $R$, we prove that $R$ is a $W$-ring if and only if, for each ideal $I$ of $R$, the ideal $\Ann(I)+\Ann^{2}(I)$ contains a regular element (Theorem~\ref{AR1}). Moreover, we show that if $R$ is a reduced $f$-ring with bounded inversion, then $R$ is $PW$ (resp., $UPW$) if and only if its bounded part is $PW$ (resp., $UPW$) (Proposition~\ref{Du}). Finally, for a reduced (resp., semiprimitive) $f$-ring with bounded inversion, we establish an equivalent condition for the co-normality of the lattice $BZ^{\circ}(R)$ (resp., $BZ(R)$) (Propositions~\ref{2} and \ref{es2}).

\section{Background and Notation}
\subsection{Rings of Continuous Functions and Topological Concepts}

In this paper, $C(X)$ (resp., $C^{*}(X)$) denotes the ring of all (resp., all bounded) real-valued continuous functions on a completely regular Hausdorff space $X$. For each $f \in C(X)$, the set $f^{-1}(\{0\})$ is called the \emph{zero-set} of $f$, and is denoted by $Z(f)$.  A $Z^{\circ}$-set in $X$ is the interior of a zero-set in $X$. A Coz$(f)$ is the set $X\setminus Z(f)$, which is called the cozero-set of $f$. The set of all open subsets of a space $X$ is denoted by $\CO(X)$. 
The space $\beta X$ is known as the \emph{Stone--\v{C}ech compactification} of $X$. It is characterized as the compactification of $X$ in which $X$ is $C^{*}$-embedded as a dense subspace. The space $\upsilon X$ is the \emph{realcompactification} of $X$, in which $X$ is $C$-embedded as a dense subspace. For a completely regular Hausdorff space $X$, we have 
\[
X \subseteq \upsilon X \subseteq \beta X.
\]
Recall from \cite{HW} that a topological space $X$ is cozero-complemented space if for each $f\in C(X)$, there is a $g\in C(X)$ such that the union of their cozero-sets is dense and the intersection of their cozero-sets is empty. 

A topological space $X$ is  called an \emph{$F$-space} when every finitely generated ideal of $C(X)$ is principal.  A  space $X$ is quasi $F$-space if each dense cozero-set of $X$ is $C^{*}$-embedded in $X$.  
We now state two useful lemmas that will be needed in the sequel.
\begin{lem}[{\cite[14.N]{GJ}}]\label{p1}
A space $X$ is an $F$-space if and only if any two disjoint cozero-sets are completely separated.
\end{lem}
\begin{lem}[{\cite[Lemma 2.10]{HVW}}]\label{p2}
A space $X$ is a quasi $F$-space if and only if any two disjoint $Z^{\circ}$-sets in $X$ have disjoint closures.
\end{lem}
\subsection{Rings} 
As mentioned in the Introduction, throughout this paper all rings are assumed to be commutative with identity. For a subset $S$ of a ring $R$, we denote by $\Ann(S)$ the annihilator of $S$ in $R$, and by $\langle S \rangle$ the ideal of $R$ generated by $S$. The set of all ideals of a ring $R$ is denoted by $\mathcal{I}(R)$. For each $a \in R$, we denote by $M_a$ (resp., $P_a$) the intersection of all maximal (resp., minimal prime) ideals of $R$ containing $a$. An ideal $I$ of a ring $R$ is called $z$-ideal (resp., $z^{\circ}$)-ideal if $M_a\subseteq I$ (resp., $P_a\subseteq I$) for each $a\in I$. The smallest $z^{\circ}$-ideal containing an ideal $I$ is denoted by $I_{\circ}$.  A ring $R$ is called \emph{reduced} if it has no nonzero nilpotent elements, and \emph{semiprimitive} if $J(R)=0$, i.e., the intersection of all maximal ideals of $R$ is zero. 

Recall that a \emph{McCoy ring} is a ring in which the annihilator of any finitely generated ideal consisting  of zerodivisors is the zero ideal.  In Huckaba's book \cite{Huc}, rings with this feature are said to satisfy Property (A). 

Recall that an \emph{$f$-ring} is a lattice-ordered ring $A$ such that for all $a,b \in A$ and $c \geq 0$, we have
\[
c(a \vee b) = (ca) \vee (cb).
\]
An element $c \in A$ is called \emph{positive} if $c \geq 0$. In particular, squares are positive in $f$-rings. An $f$-ring is said to have \emph{bounded inversion} if every element greater than $1$ is invertible. Every   $C(X)$ is a reduced $f$-ring with bounded inversion. For $a \in A$, the \emph{absolute value} of $a$, denoted by $\lvert a \rvert$, is defined as 
\[
\lvert a \rvert = a \vee (-a),
\]
which is always positive. 

In \cite{T}, it was shown that if $R$ is a reduced $f$-ring with bounded inversion, then the set 
\[
BZ(R) = \{M_f : f \in R\},
\]
partially ordered by inclusion, forms a distributive lattice with operations
\[
M_a \vee M_b = M_{a^2+b^2}, \qquad M_a \wedge M_b = M_{ab}.
\]
Moreover, the set 
\[
BZ^{\circ}(R) = \{P_f : f \in R\},
\]
partially ordered by inclusion, also forms a distributive lattice with operations
\[
P_a \vee P_b = P_{a^2+b^2}, \qquad P_a \wedge P_b = P_{ab}.
\]
Further results concerning these lattices of ideals are given in \cite{TA, T1}.

Recall from   \cite{AT}, \cite{BB1}, \cite{C}, and \cite{TA}  that a lattice $<L, \wedge, \vee, 0, 1>$ is called a co-normal lattice whenever it is a distributive lattice and for all $a, b\in L$ with $a\wedge b=0$ there exist $x, y\in L$ such that $x\vee y=1$ and $x\wedge a=y\wedge b=0$.  Trivially, every Boolean algebra is a co-normal lattice.

In this paper, we use $\Spec(R)$ (resp., $\Min(R)$) for the
spaces of prime ideals (resp., minimal prime ideals) of $R$ with the
$hull$-$kernel$ topology. For a subset $S$ of $R$, let $h(S)=\{P\in\Spec(R): S\subseteq P\}$. If $S=\{a\}$, then we use $h(a)$. The set $\{h(a): a\in R\}$ forms a base for closed sets in $\Spec(R)$. $\Min(R)$ is a subspace of $\Spec(R)$, and we use $h_{m}(S)$ instead of   $h(a)\cap\Min(R)$.
We need the following lemmas in the sequel. 
\begin{lem}\label{p3}
Let $I, J$ be two ideals of a reduced ring $R$ and $Y=\Min(R)$. 
\begin{enumerate}
\item $\Ann(I)=\Ann(J)\quad \text{\ifif} \rint_{Y}h_m(I)=\rint_{Y}h_m(J).$

\item For each $S\subseteq R$, $h_m(S)=\rint_{Y}h_m(S)$.
\end{enumerate}
\end{lem}
\begin{lem}\label{p4}
Let $R$ be a reduced ring. Then $P\in\Min(R)$ \ifif for each $a\in P$ there exists $c\not\in P$ such that $ac=0$ (i.e., $\Ann(a)\not\subseteq P$) .
\end{lem}
\section{A new extension of $F$-spaces and  cozero-complemented spaces}
Recall from \cite{AT} that a space $X$ is a $WED$-space if every two disjoint open sets in it can be separated by two disjoint $Z^{\circ}$-sets (i.e., the interior of a zero-set). Now, we extend this class of topological space to a large class. 
\begin{defn}
A topological space $X$ is said to be $WCF$-space if for every two disjoint cozero-set $A,B\in\coz(X)$, there exist $Z_1,Z_2\in Z(X)$ containing $A$ and $B$, respectively, such that $Z_1^\cc\cap Z_2^\cc=\tohi$.
\end{defn}
The above definition can also be presented in another way.

\begin{defn}
Let $\mathcal{B}, \mathcal{D} \subseteq \mathcal{P}(X)$ (the power set of $X$).  
Two distinct subsets $F,H \subseteq X$ are said to be \emph{$\mathcal{B}$-separated} if there exist two disjoint sets $A,B \in \mathcal{B}$ such that $F \subseteq A$ and $H \subseteq B$.  
We say that $\mathcal{D}$ is \emph{$\mathcal{B}$-separated} if, for every two disjoint sets $D_1,D_2 \in \mathcal{D}$, there exist disjoint $B_1,B_2 \in \mathcal{B}$ such that $D_1 \subseteq B_1$ and $D_2 \subseteq B_2$.  
A space $X$ is called \emph{$\mathcal{D}$-$\mathcal{B}$-separated} if $\mathcal{D}$ is $\mathcal{B}$-separated.  
Moreover, we say that $X$ is \emph{basically $\mathcal{B}$-separated} if there exists a base $\mathcal{D}$ for the topology of $X$ such that $\mathcal{D}$ is $\mathcal{B}$-separated.
\end{defn}
\begin{Rem}
By this definition, a space $X$ is a $WED$-space (resp., $WCF$-space) if and only if it is $\CO(X)$-$Z^\cc(X)$-separated (resp., $\coz(X)$-$Z^\cc(X)$-separated).  
In particular, a $WCF$-space $X$ is basically $Z^\cc(X)$-separated.
\end{Rem}
Since, by \cite[Lemma 2.11]{AlTaTa}, for every $f,g\in C(X)$ we have
\[
\overline{\coz(f)\cap \coz(g)}^{\,\cc}
= \overline{\coz(f)}^{\,\cc}\cap \overline{\coz(g)}^{\,\cc},
\]
the following proposition follows immediately.

\begin{prop}
A topological space $X$ is a $WCF$-space if and only if every pair of supports with disjoint interiors are $Z^\cc(X)$-separated.
\end{prop}

\begin{exam}
(1) Every $WED$-space is a $WCF$-space. In particular, every perfectly normal space (and hence every metric space) is a $WCF$-space.

(2) Every $F$-space is a $WCF$-space,  by Lemma \ref{p1}.
\end{exam}
\begin{prop}
The following statements hold.
\begin{enumerate}
\item Every cozero-complemented space is a $WCF$-space.

\item A $WCF$-space $X$  that is also a quasi $F$-space is an $F^{\prime}$-space.
\end{enumerate}
\end{prop}

\begin{proof}
(1) Let $X$ be a cozero-complemented space and let $Coz(f)$ and $Coz(g)$ be two disjoint cozero-sets in $X$.  
By hypothesis, for $f$ there exists $f_1$ and for $g$ there exists $g_1$ such that 
\[
Coz(f)\cap Coz(f_1)=\emptyset,\quad \rint Z(f)\cap \rint Z(f_1)=\emptyset,
\]
\[
Coz(g)\cap Coz(g_1)=\emptyset,\quad \rint Z(g)\cap \rint Z(g_1)=\emptyset.
\]
Now, put $f_2=f_1^2+g^2$ and $g_2=g_1^2+f^2$. Then $f_2,g_2\in C(X)$, and we have  
$Coz(f)\subseteq Z(f_1)$ and $Coz(f)\subseteq Z(g)$, hence
\[
Coz(f)\subseteq \rint Z(f_1)\cap \rint Z(g)\subseteq \rint Z(f_2)
\]
Similarly, $Coz(g)\subseteq \rint Z(g_2)$.  On the other hand,
\[
\rint Z(f_2)\cap \rint Z(g_2)
\subseteq \rint Z(g)\cap \rint Z(g_1)
=\emptyset.
\]
Thus, $X$ is a $WCF$-space.

(2) Consider two disjoint cozero-sets $Coz(f)$ and $Coz(g)$ in $X$. Then, there exist two disjoint $Z^{\circ}$-sets $\rint Z(f_1)$ and $\rint Z(g_1)$ such that $Coz(f)\subseteq \rint Z(f_1)$ and $Coz(g)\subseteq\rint Z(g_1)$. By  Lemma  \ref{p2}, $\ov{\rint Z(f_1)}\cap \ov{\rint Z(g_1)}=\emptyset$. This implies $\ov{ Coz(f)}\cap\ov{ Coz(g)}=\emptyset$, which means that $X$ is an $F^{\prime}$-space.
\end{proof}

The following example shows that the class of $WCF$-spaces properly contains the classes of $F$-spaces and cozero-complemented spaces.

\begin{exam}
Assume that $\{X_\la\}_{\la\in\La}$ is a pairwise disjoint family of topological spaces and $X$ is the free union of these spaces. It is easy to see that $X$ is a $WCF$-space (cozero-complemented space, $F$-space) if and only if $X_\la$ is a $WCF$-space (cozero-complemented space, $F$-space) for every $\la\in\La$. Now, suppose that $X$ is an $F$-space which is not a cozero-complemented space, and $Y$ is a cozero-complemented space which is not an $F$-space. Let $T$ be the free union of $X$ and $Y$. Clearly, $T$ is a $WCF$-space which is neither a cozero-complemented space nor an $F$-space.
\end{exam}

In the next example, we present a $WCF$-space which is not a $WED$-space.

\begin{exam}(\cite[14.N]{GJ})\label{fati}
Let $X$ be an uncountable space in which all points are isolated except for a distinguished point $s$. A neighborhood of $s$ is defined to be any set containing $s$ whose complement is countable. Then, $X$ is a $P$-space. Hence, $X$ is a $WCF$-space.
    Consider two disjoint uncountable open sets $A, B \subseteq X \setminus \{s\}$. Then
    $s \in \overline{A} \cap \overline{B}$.
    Suppose, for contradiction, that $X$ were a $WED$-space. Then there would exist zero-sets $Z_1, Z_2 \in Z[X]$ such that
    \[
    A \subseteq Z_1, \quad B \subseteq Z_2, \quad \text{and} \quad \operatorname{int} Z_1 \cap \operatorname{int} Z_2 = \emptyset.
    \]
    However, since $s \in \overline{A} \cap \overline{B}$, we must have $s \in Z_1 \cap Z_2$. But $\{s\}$ is not a zero-set.
    Therefore, $\operatorname{int} Z_1 \cap \operatorname{int} Z_2 \neq \emptyset$, a contradiction. Thus, $X$ is a $WCF$-space which is not a $WED$-space.
\end{exam}
Next we give an example of a non-$WCF$-space.

\begin{exam}
Let $D$ be an uncountable discrete space and let $X=D\cup\{\si\}$ be the one-point compactification of $D$.  
It is clear that a subset containing $\si$ is a zero-set if and only if its complement is countable.  
Suppose $F,H$ are two disjoint infinite countable cozero-sets in $X$, with $F\subseteq Z_1^\cc$ and $H\subseteq Z_2^\cc$.  
Obviously $\si\in \ov{F}\cap \ov{H}\subseteq Z_1\cap Z_2$.  
Hence $Z_1\cap Z_2$ is uncountable, and therefore $Z_1^\cc\cap Z_2^\cc\neq\emptyset$.  
This shows that $X$ is not a $WCF$-space.
\end{exam}

The next example shows that among spaces with only one non-isolated point, where neighborhoods of this point are determined by the cardinality of their complements, the one-point compactification is the only one that fails to be a $WCF$-space.

\begin{exam}
Let $\al$ and $\be$ be infinite cardinals with $\al<\be$.  
Assume $X=D\cup\{\si\}$ with $|X|=\be$, where each point of $D$ is isolated, and 
\[
\CO_\si=\{A\subseteq X:~\si\in A,~|X\setminus A|\leq\al\}, \text{i.e., the set of open neighborhoods of}~\sigma.\]
Then $X$ is a $P$-space, and hence $X$ is a $WCF$-space.
\end{exam}
Now we present an example of a space that is neither compact nor a $WCF$-space. 
To present that, we need the following proposition.

\begin{prop}\label{jan}
Let $X$ be a topological space with only one non-isolated point $\si$, where $\si$ is not a $G_\de$-point.  
Then $X$ is a $WCF$-space if and only if, for any two disjoint cozero-sets, one of them is a clopen subset.
\end{prop}
\begin{proof}
($\Rightarrow$)  
Suppose $A,B\in\coz(X)$ are disjoint. It suffices to show that $\si\notin \ov{A}\cap \ov{B}$.  
Assume, to the contrary, that $\si\in \ov{A}\cap \ov{B}$.  
Let $Z_1,Z_2\in Z(X)$ be such that $A\subseteq Z_1^\cc$ and $B\subseteq Z_2^\cc$.  
Thus $\si\in Z_1\cap Z_2 \in Z(X)$.  Since $\si$ is not a $G_\de$-point, the set $Z_1\cap Z_2$ must contain an isolated point.  
Hence $(Z_1\cap Z_2)^\cc\neq\emptyset$, which gives a contradiction.  

($\Leftarrow$)  
This direction is immediate.
\end{proof}

\begin{exam}
Let $Y$ be a topological space in which every countable intersection of open dense subsets is nonempty. Furthermore, suppose that $Y$ contains two disjoint dense countable subsets $A,B$ (for example, $\BBR$ with the standard topology).  
Define $X=Y\cup\{\si\}$ such that every point of $Y$ is assumed to be an isolated point of $X$, and 
\[
\CO_\si=\{\,U\cup\{\si\}:~U\in\CO(Y),\ \ov{U}=Y\,\},~ 
\text{i.e., the set of open neighborhoods of}~\sigma.\]
It is easy to see that $X$ with this topology is a completely regular Hausdorff space, $\si$ is not a $G_\de$-point, and $A,B\in\coz(X)$ with $A\cap B=\emptyset$.  
Moreover, $\si\in\ov{A}\cap\ov{B}$.  
Thus $A$ and $B$ are not clopen subsets, and by Proposition~\ref{jan}, $X$ is not a $WCF$-space.
\end{exam}

\begin{thm}\label{sara}
The following statements hold.
\begin{enumerate}
\item Let $X$ be a dense $z$-embedded subset of a space $Y$. Then $X$ is a $WCF$-space \ifif $Y$ is a $WCF$-space.

\item Let $X$ be a dense $C^*$-embedded subspace of a space $Y$. Then $X$ is a $WCF$-space \ifif $Y$ is a $WCF$-space.

\item Every cozero-set in a $WCF$-space is a $WCF$-space.
\end{enumerate}
\end{thm}

\begin{proof}
(1 $\Rightarrow$).  
Assume $X$ is a dense $WCF$-subspace of $Y$. Suppose that $A$ and $B$ are two disjoint cozero-sets in $Y$.  
Then $A\cap X$ and $B\cap X$ are disjoint cozero-sets in $X$. By hypothesis, there exist zero-sets $Z(h_1),Z(h_2)\in Z(X)$ such that
\[
A\cap X \subseteq Z(h_1), \quad B\cap X \subseteq Z(h_2),
\quad \text{and} \quad \rint_X Z(h_1)\cap \rint_X Z(h_2)=\emptyset.
\]

Since $X$ is $z$-embedded in $Y$, there exist zero-sets $Z(f_1),Z(f_2)\in Z(Y)$ such that  
$Z(h_1)=Z(f_1)\cap X$ and $Z(h_2)=Z(f_2)\cap X$.  
Thus
\[
A\cap X \subseteq Z(f_1)\quad\text{and} \quad B\cap X \subseteq Z(f_2).
\]
Since $X$ is dense in $Y$, we obtain
\[
A\subseteq \rcl_Y(A\cap X) \subseteq Z(f_1), 
\quad\text{and}\quad
B\subseteq \rcl_Y(B\cap X) \subseteq Z(f_2).
\]
If $\rint_Y Z(f_1)\cap \rint_Y Z(f_2)\neq\emptyset$, then
\[
\rint_Y Z(f_1)\cap \rint_Y Z(f_2)\cap X \neq \emptyset.
\]
But
\[
\rint_Y Z(f_1)\cap X \subseteq \rint_X(Z(f_1)\cap X)=\rint_X Z(h_1),
\]
and similarly for $Z(f_2)$.  
Hence $\rint_X Z(h_1)\cap \rint_X Z(h_2)\neq\emptyset$, a contradiction.

(1 $\Leftarrow$).  
Assume $X$ is a dense $z$-embedded subspace of a $WCF$-space $Y$.  
Let $A,B$ be two disjoint cozero-sets in $X$.  
Since $X$ is $z$-embedded in $Y$, there exist cozero-sets $A',B'$ in $Y$ such that  
$A' \cap X=A$ and $B' \cap X=B$.  
Since $X$ is dense in $Y$, we have $A'\cap B'=\emptyset$. By hypothesis, there exist $Z(f_1),Z(f_2)\in Z(Y)$ such that
\[
A' \subseteq Z(f_1), \quad B' \subseteq Z(f_2), 
\quad \text{and} \quad \rint_Y Z(f_1)\cap \rint_Y Z(f_2)=\emptyset.
\]
Thus
\[
A\subseteq Z(f_1)\cap X,\quad\text{and} \qquad B\subseteq Z(f_2)\cap X.
\]
Suppose, toward a contradiction, that
\[
\rint_X(Z(f_1)\cap X)\cap \rint_X(Z(f_2)\cap X)\neq\emptyset.
\]
Then, there exists $x\in X$ and open sets $G\subseteq Y$ such that  $x\in G\cap X \subseteq Z(f_1)\cap Z(f_2)$. Since $X$ is dense in $Y$, we have $x\in G \subseteq \rcl_Y(G\cap X)\subseteq Z(f_1)\cap Z(f_2)$. Hence, $x\in \rint_Y Z(f_1)\cap \rint_Y Z(f_2)$, contradicting the choice of $f_1,f_2$.  

(2). Every $C^*$-embedded (and hence $C$-embedded) subspace is $z$-embedded.  
Thus the result follows directly from (1).  

(3).  
Let $X$ be a cozero-set in a $WCF$-space $Y$.  
By \cite[Proposition~1.1]{BH}, $X$ is $z$-embedded in $Y$.  
Since the property of being a cozero-set is transitive, two disjoint cozero-sets $A,B$ in $X$ are also disjoint cozero-sets in $Y$.  Hence there exist $Z(f_1),Z(f_2)\in Z(Y)$ such that
\[
A\subseteq Z(f_1), \quad B\subseteq Z(f_2), 
\quad \text{and} \quad \rint_Y Z(f_1)\cap \rint_Y Z(f_2)=\emptyset.
\]
Thus
\[
A\subseteq Z(f_1)\cap X,\quad\text{and} \qquad B\subseteq Z(f_2)\cap X.
\]
Since $X$ is open in $Y$, we have
\[
\rint_X(Z(f_1)\cap X)\cap \rint_X(Z(f_2)\cap X) 
= \big(\rint_Y Z(f_1)\cap X\big)\cap \big(\rint_Y Z(f_2)\cap X\big)=\emptyset.
\]
Therefore $X$ is a $WCF$-space.  
\end{proof}

\begin{cor}\label{1}
The following statements hold.
\begin{enumerate}
\item A space $X$ is a $WCF$-space \ifif $\beta X$ is a $WCF$-space.

\item Let $X\subseteq Y\subseteq\beta X$. Then $X$ is a $WCF$-space \ifif $Y$ is a $WCF$-space.

\item A  space $X$ is a $WCF$-space \ifif $\upsilon X$ is a $WCF$-space.
\end{enumerate}
\end{cor}
\begin{proof}
(1)  This follows from Part 2 of Theorem \ref{sara}.

(2) By \cite[Theorem 6.7]{GJ}, $\beta Y=\beta X$. Thus, by Part(1), $X$ is a $WCF$-space \ifif $\beta Y$ is so, and so again by Part (1), $X$ is a $WCF$-space \ifif $Y$ is so.

(3) This follows from Part (2).
\end{proof}

\begin{prop}\label{}
The following statements hold.
\begin{enumerate}
\item If a space $X$ is a $WED$-space and $Y\in\CO(X)$ (i.e., the open subsets of $X$), then $Y$ is also a $WED$-space.

\item  If a  space $X$ is a $WED$-space and $Y\in\coz(X)$, then $Y$ is also a $WED$-space.
\end{enumerate}
\end{prop}
\begin{proof}
(1). Suppose that $U,V\in\CO(Y)$ and $U\cap V=\tohi$. By hypothesis, $U,V\in\CO(X)$ and so there exist $Z_1,Z_2\in Z(X)$ such that $U\sub Z_1^\cc$, $V\sub Z_2^\cc$, and $Z_1^\cc\cap Z_2^\cc=\tohi$. Clearly, $A=Z_1\cap Y\in Z(Y)$, $B=Z_2\cap Y\in Z(Y)$, and also we have:
\[U\sub Z_1^\cc\cap Y=(Z_1\cap Y)^\cc=A^\cc=\rint_Y(A)~~,\]
\[V\sub Z_2^\cc\cap Y=(Z_2\cap Y)^\cc=B^\cc=\rint_Y(B),\text{and}\quad\rint_Y(A)\cap\rint_Y(B)=\tohi.\]

(2). This follows from Part (1).
\end{proof}

\begin{prop}\label{}
Suppose that for each $\la\in\La$, $X_\la$ is a topological space and $X=\prod_{\la\in\La}X_\la$. Then the following statements hold.
\begin{enumerate}
\item  Supposing that 
\[\CD=\{\bcap_{\la\in F}\pi_\la^{-1}(A_\la):~F~\text{is a finite subset of}~\La~\text{and}~A_\la\in\CO(X_\la)\}.\]
If $X_\la$ is a $WED$-space for every $\la\in\La$, then $X$ is a $\CD$-$Z^\cc(X)$-space.

\item Supposing that 
\[\CD=\{\bcap_{\la\in F}\pi_\la^{-1}(A_\la):~F~\text{is a finite subset of}~\La~\text{and}~A_\la\in\coz(X_\la)\}.\]
If $X_\la$ is a $WCF$-space for every $\la\in\La$, then $X$ is a $\CD$-$Z^\cc(X)$-space.
\end{enumerate}
\end{prop}
\begin{proof}
(1). Assume that $U=\bcap_{\la\in F_1}\pi_\la^{-1}(A_\la), V=\bcap_{\la\in F_2}\pi_\la^{-1}(B_\la)\in\CD$ such that $U\cap V=\tohi$ where $F_1$ and $F_2$ are two finite subset of $\La$, and $A_\la,B_\la\in\CO(X_\la)$. If $U=\tohi$ or $V=\tohi$, then we have nothing to do. Otherwise, there exists $\ga\in F_1\cap F_2$ such that $A_\ga\cap B_\ga=\tohi$ and so there exist $Z_1,Z_2\in Z(X_\ga)$ containing $A_\ga$ and $B_\ga$ respectively, and $Z_1^\cc\cap Z_2^\cc=\tohi$. Put $T_1=\pi^{-1}Z_1$ and $T_2=\pi^{-1}Z_2$. It is easy to see that $T_1,T_2\in Z(X)$, $U\sub T_1$, $V\sub T_2$, and $T_1^\cc\cap T_2^\cc=\tohi$.

(2). It is  similar to the Part (1).
\end{proof}

\section{Algebraic characterization of $WCF$-spaces}

Recall from \cite{AT} that a ring $R$ is $WSA$ if for each two ideals $I, J$ of $R$ with $I\cap J=0$, we have $(\Ann(I)+\Ann(J))_{\circ}=R$. The class of $WSA$-rings containing the class of $SA$-rings. In \cite{DT}, the authors defined an $f$-ring $R$ to be  $\emph{wedded}$ if for every pair of annihilator ideals  $I, J$ of $R$ with $I\cap J=0$, $\Ann(I)+\Ann(J)$ contains a non-zero-divisor element. They further defined an $f$-ring $R$ to be  $\emph{strongly wedded}$ if for every pair of  ideals  $I, J$ of $R$ with $I\cap J=0$, the sum $\Ann(I)+\Ann(J)$ contains a non-zero-divisor element, after that, in Lemma 1.4 of the same paper, they proved   that a reduced $f$-ring is strongly wedded \ifif it is wedded. We now propose a generalization of this concept as follows:

\begin{defn}
  A ring $R$ is called a \emph{$W$-ring} (resp., \emph{$UW$-ring}) if for each pair of ideals $I, J$ of $R$ with $I\cap J=0$, the sum $\Ann(I)+\Ann(J)$ contains a regular element (resp., unit element).
\end{defn}
We recall some well-known results here to use them in the sequel.
\begin{lem}\label{nem}
The following statements hold.
\begin{enumerate}
\item (\cite[Theorem 4.12]{DT}) A reduced McCoy $f$-ring is wedded \ifif it is a $WSA$-ring.

\item (\cite[Theorem 3.8]{AT}) $C(X)$ is a $WSA$-ring \ifif $X$ is a $WED$-space.

\item (\cite[Corllary 2.13]{T}) $C(X)$ is a $UW$-ring \ifif $X$ is an extremally disconnected.

\item $C(X)$  is a $W$-ring \ifif $X$ is a $WED$-space.
\end{enumerate}
\end{lem}
 \vspace{3mm}

Now, we introduce a large class of rings which contains $W$-rings.
 
\begin{defn}
A ring $R$ is called a \emph{principally wedded ring}, abbreviated as \emph{$PW$-ring}, (resp., a \emph{$UPW$-ring}) if for every $a,b \in R$ with $aR \cap bR = 0$, the ideal $\Ann(a) + \Ann(b)$ contains a regular element (resp., a unit element).  
More generally, let $\mathcal{D} \subseteq \mathcal{I}(R)$, where $\mathcal{I}(R)$ denotes the set of all ideals of $R$. We say that $R$ is a \emph{$\mathcal{D}$-$W$-ring} (resp., a \emph{$\mathcal{D}$-$UW$-ring}) if for every pair of ideals $I,J \in \mathcal{D}$ with $I \cap J = 0$, the ideal $\Ann(I) + \Ann(J)$ contains a regular element (resp., a unit element).
\end{defn}
It is clear that if $\CD$ is the set of all principal ideals of $R$, then $\CD$-$W$-ring (resp., $\CD$-$UW$-ring) is the same as $PW$-ring (resp., $UPW$-ring).

By definitions, we have these implications: \[ W\text{-ring} \to PW\text{-ring},\quad UW\text{-ring} \to W\text{-ring, and}\quad UPW\text{-ring} \to PW\text{-ring}.\] However, we will see that the converses do not necessarily hold.

\begin{prop}\label{has}
Suppose that for each $\la\in \La$, $R_\la$ is a ring and $R=\prod_{\la\in\La}R_\la$. Then, $R$ is a $PW$-ring ($UPW$-ring) if and only if $R_\la$ is so for every $\la\in\La$.
\end{prop}
\begin{proof}
We prove it for $PW$-ring, the proof for $UPW$-ring is similarly. 

($\RTA$).  Suppose that $\ga$ is an arbitrary element of $\La$, $a_\ga,b_\ga\in R_\ga$ with $a_\ga R_\ga\cap b_\ga R_\ga=0$. Take $x,y\in R$ such that $x_\la=y_\la=0$ for every $\la\neq\ga$, $x_\ga=a_\ga$, and $y_\ga=b_\ga$. Clearly, $xR\cap yR=0$ and so $\Ann(x)+\Ann(y)$ contains a regular element. It is easy to see that; 
$$\Ann(x)+\Ann(y)=\prod_{\la\in\La}\Ann(x_\la)+\prod_{\la\in\La}\Ann(y_\la)$$
$$=\prod_{\la\in\La}(\Ann(x_\la)+\Ann(y_\la))=(\Ann(a_\ga)+\Ann(b_\ga))\times\prod_{\la\in\La\set\{\ga\}}R_\la.$$
Therefore, since $\Ann(x)+\Ann(y)$ contains a regular element, it follows that $\Ann(a_\ga)+\Ann(b_\ga)$ also contains  a regular element.

( $\LTA$). Suppose $a,b\in R$ with $aR\cap bR=0$. It is easily seen that $aR\cap bR=\prod_{\la\in\La}(a_\la R_\la\cap b_\la R_\la)$. Thus, for every $\la\in\La$, $a_\la R_\la\cap b_\la R_\la=0$ and so $\Ann(a_\ga)+\Ann(b_\ga)$ contains a regular element. Consequently, $\Ann(x)+\Ann(y)=\prod_{\la\in\La}(\Ann(x_\la)+\Ann(y_\la))$ contains a regular element.
\end{proof}
\vspace{3mm}
For the next proposition, we need the following well-known lemma. Let $\{R_\lambda:~\lambda \in \Lambda\}$ be a family of rings, and let $R = \prod_{\lambda \in \Lambda} R_\lambda$.

For any $x \in R$, we denote by $x_\lambda$ the $\lambda$-component of $x$, i.e., $x_\lambda = \pi_\lambda(x)$, where $\pi_\lambda : R \to R_\lambda$ is the canonical projection.  
For any $a_\gamma \in R_\gamma$, we denote by $a_\gamma^{c}$ the element of $R$ defined by  
\[
(a_\gamma^{c})_\lambda =
\begin{cases}
a_\gamma, & \lambda = \gamma, \\
0, & \lambda \neq \gamma .
\end{cases}.\]
For any ideal $I \in \mathcal{I}(R)$, we denote by $I_\lambda$ the projection of $I$ onto $R_\lambda$, i.e., $I_\lambda = \pi_\lambda(I)$.  
For any ideal $I_\gamma \in \mathcal{I}(R_\gamma)$, we denote by $I_\gamma^{c}$ the ideal of $R$ defined by  
\[
(I_\gamma^{c})_\lambda =
\begin{cases}
I_\gamma, & \lambda = \gamma, \\
(0), & \lambda \neq \gamma .
\end{cases}
\]
\begin{lem}\label{42}
Let $\{R_\lambda:~\lambda \in \Lambda\}$ be a family of rings, and let $R=\prod_{\lambda\in\Lambda}R_\lambda$. Then the following statements hold:

\begin{enumerate}
    \item For every $I\in\mathcal{I}(R)$, we have 
    \[
    I \subseteq \prod_{\lambda\in\Lambda} I_\lambda.
    \]
    
    \item If $I_\lambda$ is an ideal of $R_\lambda$ for every $\lambda\in\Lambda$, then 
    \[
    \Ann\!\left(\prod_{\lambda\in\Lambda} I_\lambda\right)
    = \prod_{\lambda\in\Lambda}\Ann(I_\lambda).
    \]
    
    \item If $I$ is an arbitrary ideal of $R$, then 
    \[
    I_\lambda^{c} \subseteq I \quad \text{for every } \lambda \in \Lambda.
    \]
    
    \item An element $x\in R$ is regular if and only if each $x_\lambda$ is a regular element of $R_\lambda$.
    
    \item An element $x\in R$ is a unit if and only if each $x_\lambda$ is a unit element of $R_\lambda$.
\end{enumerate}
\end{lem}

\begin{proof}
Each assertion follows directly from the definitions and the componentwise structure of $R=\prod_{\lambda\in\Lambda}R_\lambda$.
\end{proof}
\begin{prop} Suppose that $R_\la$ is a ring for every $\la\in\La$, $R=\prod_{\la\in\La}R_\la$, and $\CD=\{\prod_{\la\in\La}I_\la: I_\la$ is an ideal of $R_\la\}$. Then the following statements are equivalent: 
\begin{enumerate}
\item $R$ is a $W$-ring (resp., $UW$-ring). 

\item $R$ is a $\CD$-$W$-ring (resp., $\CD$-$UW$-ring). 

\item $R_\la$ is a $W$-ring (resp., $UW$-ring) for every $\la\in\La$. 
\end{enumerate}
\end{prop} 
\begin{proof} (1) $\RTA$ (2). It is clear. 

(2) $\RTA$ (3). Let $I_\ga, J_\ga\in\CI(R_\ga)$ be such that $I_\ga\cap J_\ga=(0)$. Hence, \[H=I_\ga^{c}, K=J_\ga^{c}\in\CD,\quad \text{and}\quad H\cap K=(0).\] By assumption , $\Ann(H)+\Ann(K)$ contains a regular element. Therefore, $(\Ann(H)+\Ann(K))_\ga=\Ann(I_\ga)+\Ann(J_\ga)$ contains a regular element. 

(3) $\RTA$ (1). Let $I, J\in\CI(R)$ be such that $I\cap J=(0)$. Since, by Lemma \ref{42}, $I_{\la}^{c}\cap J_\la^{c}\sub I\cap J=(0)$ for every $\la\in\La$, it follows that $I_\la\cap J_\la=(0)$ for every $\la\in\La$. Hence, $\Ann(I_\la)+\Ann(J_\la)$, for every $\la\in\La$, contains a regular element $r_\la$. Thus, by Lemma \ref{42}, \[\Ann(\prod_{\la\in\La}I_\la)+\Ann(\prod_{\la\in\La}J_\la)=\prod_{\la\in\La}(\Ann(I_\la)+\Ann(J_\la))\] contains the regular element $r=(r_\la)\in R$. On the other hand, we have \[\Ann(\prod_{\la\in\La}I_\la)+\Ann(\prod_{\la\in\La}J_\la)\sub\Ann(I)+\Ann(J).\] Therefore, $\Ann(I)+\Ann(J)$ contains the regular element $r$.

The argument for $UW$-rings is entirely analogous.
\end{proof}
Recall from \cite{C} that a ring $R$ is real \ifif  for all $n\in\Bbb N$: \[a_1^2+ a_2^2+...+ a_n^2=0 \Leftrightarrow a_1=a_2=...=a_n=0 \quad(\forall a_1, ..., a_n\in R).\] 
For $n\in\Bbb N$, let us call a ring $R$ an $n$-real ring if for each $a_1, ..., a_n\in R$, the equality $a_1^2+ a_2^2+...+ a_n^2=0$ implies $a_1=a_2=...= a_n=0$. Evidently, a ring $R$ is real \ifif it is $n$-real for each $n\in\Bbb N$. Now, we give a characterization of real rings.

\begin{lem}\label{AR}
A ring $R$ is real \ifif $R$ is reduced and for each $n\in\Bbb N$ and for all $a_1, ..., a_n\in R$, we have,    $\bigcap_{i=1}^{n}h_m(a_i)=h_m(a_1^2+ a_2^2+ ... + a_n^2)$.
\end{lem}
\begin{proof}
$\Rightarrow$ Evidently, $R$ is a reduced ring. Always we have: \[\bigcap_{i=1}^{n}h_m(a_i)\subseteq h_m(a_1^2+ a_2^2+ ... + a_n^2).\] Now, assume $P\in h_m(a_1^2+ a_2^2+ ... + a_n^2)$. Then, by Lemma  \ref{p4}, there exists $c\notin P$ such that $c(a_1^2+ a_2^2+ ... + a_n^2)=0$. Thus, \[(ca_1)^2+ (ca_2)^2+...+ (ca_{n})^2=0.\] By hypothesis, $ca_n=0$, for each $n\in\Bbb N$, and hence $ca_n\in P$. This implies $a_n\in P$, for each $n\in\Bbb N$, i.e., $P\in \bigcap_{n\in\Bbb N}h_m(a_n)$. So, we are done.

$\Leftarrow$  Let $n\in\Bbb N$ and $a_1, a_2,..., a_n\in R$ with $a_1^2+ a_2^2+...+ a_n^2=0$. Then, \[h_m(a_1)\cap h_m(a_2)\cap ... \cap h_m(a_n)= h_m(a_1^2+ a_2^2+...+ a_n^2)=h_m(0)=\Min(R).\] This implies that $h_m(a_1)=h_m(a_2)=...=h_m(a_n)=\Min(R)$.  Since, $R$ is a reduced ring, $a_1= a_2=...= a_n=0$.
\end{proof}
\begin{thm}\label{AR1}
Let $R$ be a reduced ring. The following statements hold.
\begin{enumerate}
\item $R$ is  a $W$ (resp., $UW$)-ring \ifif for each  ideal $I$ of $R$, $\Ann(I)+\Ann^{2}(I)$ contains a regular element (resp., a unit element).
 
\item If $R$ is a $W$-ring, then every two disjoint open sets in $\Min(R)$ can be separated by two disjoint basic closed elements. 

\item The converse of Part (2) for  a real ring ($2$-real) holds. 
\end{enumerate}
\end{thm}
\begin{proof}
(1 $\RTA$). Since $R$ is a reduced ring, for each ideal $I$ of $R$, $I\cap \Ann(I)=0$. By hypothesis, $\Ann(I)+\Ann^{2}(I)$ contains a regular element (resp., a unit element).

(1 $\LTA$). Let $I, J$ be two ideals of $R$ with $I\cap J=0$. Then $J\subseteq\Ann(I)$ and hence $\Ann^{2}(I)\subseteq\Ann(J)$. Thus, \[\Ann(I)+\Ann^{2}(I)\subseteq \Ann(I)+\Ann(J).\] By hypothesis, $\Ann(I)+\Ann^{2}(I)$ contains a regular element (resp., a unit element). Therefore, $\Ann(I)+\Ann(J)$ contains a regular element (resp., a unit element).

(2)  Let $A, B$ be two disjoint open sets in $Y=\Min(R)$. Then,  there are two subsets $S, K$ of $R$ such that \[A=\bigcup_{a\in S}h_m^c(a) \quad\text{and}\quad B=\bigcup_{b\in K}h_m^c(b).\] Put $I=<S>$ and $J=<K>$. Since $A\cap B=\emptyset$, it follows that for each $a\in S$,  \[h_m^c(aJ)=h_m^c(a)\cap h_m^c(J)=h_m^c(a)\cap h_m^c(K)=\emptyset,\] and hence $aJ=0$. Therefore, $IJ=0$ and so $I\cap J=0$. Thus, by hypothesis, $\Ann(I)+\Ann(J)$ contains a regular element say $c$. Hence, there are $x\in \Ann(I)$ and $y\in\Ann(J)$ such that $c=x+y$. The regularity of $c$ implies that $\Ann(c)=0$, thus by Lemma \ref{p3}, $h_m(c)=\rint_Yh_m(c)=\emptyset$ and so $h_m(x)\cap h_m(y)=\emptyset$. On the other hand, $x\in\Ann(I)$ and $y\in\Ann(J)$ imply $xI=0$ and $yJ=0$. Thus we have:
\[A=\bigcup_{a\in S}h_m^c(a)=h_m^c(I)\subseteq h_m(x)\quad\text{and}\quad B=\bigcup_{b\in K}h_m^c(b)=h_m^c(J)\subseteq h_m(y).\]
(3) Let $I, J$ be two  ideals of $R$ with $I\cap J=0$. Then $IJ=0$ and hence, \[h_m^c(I)\cap h_m^c(J)=h_m^c(IJ)=\emptyset.\] By hypothesis, there are $a, b\in R$ such that $h_m^c(I)\subseteq h_m(a)$, $h_m^c(J)\subseteq h_m(b)$ and $h_m(a)\cap h_m(b)=\emptyset$. Thus, \[h_m^c(Ia)=h_m^c(I)\cap h_m^c(a)=\emptyset,\quad  h_m^c(Jb)=h_m^c(J)\cap h_m^c(b)=\emptyset,\] and $h_m(a)\cap h_m(b)=\emptyset$. Hence, $aI=0$, $bJ=0$ and \[\rint_Yh_m(a^2+ b^2)=h_m(a^2+ b^2)=h_m(a)\cap h_m(b)=\emptyset,\] by Lemmas \ref{p3} and  \ref{AR}. Therefore, $a\in\Ann(I)$, $b\in\Ann(J)$ and $\Ann(a^2+ b^2)=0$, i.e., $a^2+ b^2$ is a regular element in $\Ann(I)+\Ann(J)$.
\end{proof}
The following result shows that the class of $UPW$-rings is a subclass of reduced rings.
\begin{prop}\label{dava}
The following statements are equivalent for any ring $R$.
\begin{enumerate}
\item $R$ is  a $UPW$-ring 

\item For each $a, b\in R$, we have $\Ann(a)+\Ann(b)=\Ann(ab)$.

\item $R$  is a reduced ring and any two disjoint basic open elements in $\Spec(R)$ have disjoint closure.
\end{enumerate}
\end{prop}
\begin{proof}
(1)$\Rightarrow$(2)  Evidently, for each $a, b\in R$, $\Ann(a)+\Ann(b)\subseteq\Ann(ab)$. Now, let $x\in\Ann(ab)$. Then $xab=0$. By hypothesis, $\Ann(ax)+\Ann(b)=R$. Thus $1=y+z^{(1)}$, where $y\in\Ann(ax)$ and $z\in\Ann(b)$. Thus, $yx\in\Ann(a)$ and $xz\in\Ann(b)$. By multiplying  the equality $(1)$ by $x$,  $x=xy+ xz\in\Ann(a)+\Ann(b)$. So we done.

(2)$\Rightarrow$(1) This is obvious.

(1)$\Rightarrow$(3) Let $a\in R$ and $a^2=0$. By Part (2) , we have: \[\Ann(a)=\Ann(a)+\Ann(a)=\Ann(a^2)=R.\] Thus, $a=0$. For the remainder of the proof see Proposition 2.17 in \cite{AlTaTa}.

(3)$\Rightarrow$(1) This follows from Part (2) and the Proposition 2.17 in \cite{AlTaTa}.
\end{proof}
Theorem 3.6 in \cite{AlTaTa} implies the next result.
\begin{cor}\label{exa}
$C(X)$ is a $UPW$-ring \ifif $X$ is an $F$-space. 
\end{cor}
Thus,  whenever $X$ is a non-$F$-space, $C(X)$ is a reduced ring which is not a $UPW$-ring.

We recall that, for an ideal $I$ of  a reduced ring $R$ with strong annihilator condition (i.e., $s.a.c$-property, e.g., $C(X)$), we have: \[I_{\circ}=\{a\in R: \exists b\in I,\quad\text{such that}\quad \Ann(b)\subseteq\Ann(a)\}.\] It is also useful to note that for a reduced ring $R$ and $a\in R$ we have: \[P_a=\{x\in R: \Ann(a)\subseteq\Ann(x)\}.\] Next result shows that the class of $UPW$-rings (hence $PW$-rings) is very large.
\begin{prop}
The following statements hold.
\begin{enumerate}
\item Every  $WSA$-ring with $s.a.c$-property  is a $PW$-ring. 

\item Every $PP$-ring is a $UPW$-ring.

\item Every reduced $IN$-ring (hence every  Baer ring) is a $UPW$-ring.

\end{enumerate}
\end{prop}
\begin{proof}
(1) Let $a, b\in R$ and  $Ra\cap Rb=0$. Then, by hypothesis, $(\Ann(a) +\Ann(b))_{\circ}=R$. Hence, $1\in (\Ann(a)+\Ann(b))_{\circ}$. This and the above comment imply the existence of and element $c\in\Ann(a)+ \Ann(b)$ with $\Ann(c)=0$. 

(2) Let $a, b\in R$. By hypothesis, there are idempotents $e, f\in R$ such that $\Ann(a)=eR$ and $\Ann(b)=fR$. Thus, $\Ann(a)+\Ann(b)= eR+ fR=(e+f-ef)R$. It is easy to see that $(e+f-ef)R=\Ann(ab)$. Hence $\Ann(a)+\Ann(b)=\Ann(ab)$. So we are done, by Proposition \ref{dava}.

(3) This follows from Theorem 2.14 in \cite{AlTaTa}.
\end{proof}
\begin{prop}\label{Du}
Let $R$ be a reduced $f$-ring with bounded inversion. Then:
\begin{enumerate}
\item $R$ is a $PW$-ring if and only if its bounded part $R^*$ is a $PW$-ring.
\item $R$ is a $UPW$-ring if and only if its bounded part $R^*$ is a $UPW$-ring.
\end{enumerate}
\end{prop}
\begin{proof}
 1) $\RTA$. 
Assume that $R$ is $PW$, $f\in R^*$ and $g\in\Ann_{R^*}(f)$. Since $R^*\sub R$, we have $\Ann_{R^*}(S)\sub\Ann_R(S)$, for every $S\subseteq R^*$ and consequently there exist $h_1\in\Ann_R(f)$ and $h_2\in\Ann_R(g)$ such that $h=h_1^2+h_2^2$ is a regular element of $R$, i.e., $\Ann_R(h)=0$. Set $u=h_1^2+h_2^2+1$. Since $u\geq1$, $\frac{1}{u}\in R$. It is clear that $\frac{h_1^2}{u}\in\Ann_{R^*}(f)$, $\frac{h_2^2}{u}\in\Ann_{R^*}(g)$, and their sum $\frac{h_1^2}{u}+\frac{h_2^2}{u}$ is a regular element in $R^*$. Thus $R^*$ is $PW$.

(1$\LTA$. Suppose that $R^*$ is $PW$, $f\in R$ and $g\in\Ann(f)$. Then \[f/(1+|f|), \quad  g/(1+|g|)\in R^*.\] Moreover, \[g/(1+|g|)\in\Ann_{R^*}(f/(1+|f|)).\] By hypothesis, there exist $h_1\in\Ann_{R^*}(f/(1+|f|))$ and $h_2\in\Ann_{R^*}(g/(1+|g|)$ such that $\Ann_{R^*}(h_1^2+ h_2^2)=0$. This implies that \[h_1\in \Ann(f),  h_2\in\Ann(g)\quad \text{and}\quad \Ann(h_1^2+ h_2^2)=0.\] So we are done.

(2$\RTA$). Suppose $R$ is $UPW$. Let $f\in R^*$ and $g\in\Ann_{R^*}(f)$. Then $g\in \Ann(f)$. By hypothesis, there exist $h_1\in\Ann(f)$ and $h_2\in\Ann(g)$ such that $h=h_1+ h_2$ is a unit element in $R$. Set $v= 1+|h_1|+ |h_1|$. Then,  $h_1/ v\in\Ann_{R^*}(f)$ and $h_2/v\in\Ann_{R^*}(g)$ and $(h_1 + h_2)/v$ is a unit element in $R^*$. Thus, $\Ann_{R^*}(f)+\Ann_{R^*}(g)=R^*$, showing that $R^*$ is a $UPW$.

(2$\LTA$). Assume that $R^*$ is $UPW$, and let $f\in R$ and $g\in\Ann(f)$. Then \[f/(1+|f|), g/(1+|g|)\in R^*\]with\[g/(1+|g|)\in\Ann_{R^*}(f/(1+|f|)).\] By assumption, there exist $h_1\in\Ann_{R^*}(f/(1+|f|))$ and $h_2\in\Ann_{R^*}(g/(1+|g|)$ such that $h_1+ h_2$ is unit in $R^*$. Clearly, $h_1\in \Ann(f)$,  $h_2\in\Ann(g)$ and $h_1+ h_2$ is unit in $R$. Hence, $R$ is $UPW$.
\end{proof}
\begin{thm}\label{sot} The following statements are equivalent. 
\begin{enumerate} 
\item $C(X)$ is $PW$. 
\item The space $X$ is a $WCF$-space. \item $C^*(X)$ is $PW$. 
\end{enumerate} 
\end{thm} \begin{proof} 
(1)$\Rightarrow$(2) Let $Coz(f)$ and $Coz(g)$ be two disjoint cozero-sets in $X$. Then $fg=0$. By hypothesis, $\Ann(f)+\Ann(g)$ contains a regular element, i.e., there exists $h\in \Ann(f)+\Ann(g)$ such that $\Ann(h)=0$. Thus, there are $h_1\in\Ann(f)$ and $h_2\in\Ann(g)$ such that $h=h_1+ h_2$. This implies that: \[\rint Z(h_1^2+ h_2^2)=\rint Z(h_1)\cap\rint Z(h_2)\subseteq\rint Z(h_1+ h_2)=\emptyset.\] On the other hand, $h_1\in \Ann(f)$ implies that $h_1f=0$, i.e., $Coz(f)\subseteq Z(h_1)$ and similarly $Coz(g)\subseteq Z(h_2)$. Therefore, $X$ is a $WCF$-space.

 (2)$\Rightarrow$(1) Let $f, g\in C(X)$ with $fg=0$. Then $Coz(f)\cap Coz(g)=\emptyset$. By hypothesis, there are two zero-sets $Z(f_1)$ and $Z(f_2)$ such that $Coz(f)\subseteq\rint Z(f_1)$, $Coz(g)\subseteq\rint Z(g_1)$ and $\rint Z(f_1^2+ f_2^2)=\rint Z(f_1)\cap\rint Z(g_1)=\emptyset$. Thus, $ff_1=0$, $gg_1=0$ and $\Ann(f_1^2+ g_1^2)=0$. Hence, $f_1^{2} +g_1^{2}\in \Ann(f)+\Ann(g)$ and $\Ann(f_1^2+ g_1^2)=0$. Thus, $C(X)$ is a $PW$-ring. 
 
 (1)$\Leftrightarrow$(3) This follows from Proposition \ref{Du}. 
\end{proof} 
Now, we want to characterize the co-normality of the lattice $BZ^{\circ}(R)$ in the class of reduced $f$-rings with bounded inversion. 
\begin{prop}\label{2} 
Let $R$ be a reduced $f$-ring with bounded inversion. Then the lattice $BZ^{\circ}(R)$ is co-normal \ifif $R$ is $PW$-wedded. 
\end{prop}
\begin{proof} $\Rightarrow$ Let $a, b\in R$ with $ab=0$. Then $P_a\wedge P_b=P_a\cap P_b=P_{ab}=P_{0}=0$, since $R$ is a reduced ring. By the hypothesis, there are $c, d\in R$ such that: \[P_a\wedge P_c=P_b\wedge P_d=0\quad\text{and}\quad P_c\vee P_d=1.\] This implies $P_{ac}=P_{bd}=0$, i.e., $ac=bd=0$ and $P_{c^2+ d^2}=1$. Hence, $c^2\in\Ann(a)$, $d^2\in\Ann(b)$ and $P_{c^2+ d^2}=1$ implies $\Ann(c^2+ d^2)=0$. Thus, $c^2+ d^2\in \Ann(a)+ \Ann(b)$ is a regular element. 

$\Leftarrow$ Consider two elements $P_a, P_b$ of $BZ^{\circ}(R)$ with $P_a\wedge P_b=0$. Then $P_{ab}=P_a\cap P_b=0$ and hence $ab=0$ in $R$. By hypothesis, $\Ann(a)+ \Ann(b)$ contains a regular element. Hence, there are $c\in\Ann(a)$ and $d\in\Ann(b)$ such that $\Ann(c+ d)=0$. Thus, $ac=0$, $bd=0$ and $\Ann(c^2+ d^2)=\Ann(c)\cap\Ann(d)\subseteq\Ann(c+ d)=0$. This implies $P_a\wedge P_c=P_a\cap P_c=P_0=0$, $P_b\wedge P_d=P_b\cap P_d=P_0=0$ and $P_c\vee P_d=P_{c^2+ d^2}=1$. Therefore, $BZ^{\circ}(R)$ is a co-normal lattice. \end{proof}
 \begin{prop}\label{es2}
 Let $R$ be a semiprimitive $f$-ring  with bounded inversion. Then the lattice $BZ(R)$ is co-normal \ifif $R$ is $UPW$.
 \end{prop}
 \begin{proof}
 $\Rightarrow$ Let $a, b\in R$ with $ab=0$. Then $M_a\wedge M_b=M_a\cap M_b=M_{ab}=M_{0}=0$, since $R$ is a semiprimitive ring. By the hypothesis, there are $c, d\in R$ such that: \[M_a\wedge M_c=M_b\wedge M_d=0\quad\text{and}\quad M_c\vee M_d=1.\] This implies $M_{ac}=M_{bd}=0$, i.e.,  $ac=bd=0$ and $M_{c^2+ d^2}=1$. Hence, $c^2\in\Ann(a)$, $d^2\in\Ann(b)$ and $M_{c^2+ d^2}=1$. Thus, $c^2+ d^2\in \Ann(a)+ \Ann(b)$ is a unit  element, i.e., $R$ is a $UW$-ring.
 
 $\Leftarrow$ Consider two elements $M_a, M_b$ of $BZ(R)$ with $M_a\wedge M_b=0$. Then $M_{ab}=M_a\cap M_b=0$ and hence $ab=0$ in $R$. By hypothesis, $\Ann(a)+ \Ann(b)$ contains a unit element. Hence, there are $c\in\Ann(a)$ and $d\in\Ann(b)$ such that $c+ d$ is unit. Thus, $ac=0$, $bd=0$ and $c+d$ is unit. This implies that: \[M_a\wedge M_c=M_a\cap M_c=M_{ac}=M_0=0,\] \[M_b\wedge M_d=M_b\cap M_d=M_{bd}=M_0=0,\quad\text{and}\]  \[M_c\vee M_d=M_{c^2+ d^2}=M_{c+d}=R.\]Therefore, $BZ(R)$ is a co-normal lattice.
 \end{proof}
 From Theorems \ref{sot}, Theorem 3.6 in \cite{AlTaTa},  Propositions \ref{2} and \ref{es2}, we deduce the next result.
 \begin{cor}
 Let $X$ be a completely regular Hausdorff space. 
 \begin{enumerate}
 \item The lattice $BZ^{\circ}(C(X))$ is co-normal \ifif $X$ is a $WCF$-space.
 
 \item The lattice $BZ(C(X))$ is co-normal \ifif $X$ is an $F$-space.
 \end{enumerate}
 \end{cor}
We conclude this section with the following example.
\begin{exam}
(1) A $W$-ring need not to be a $UW$-ring. Consider a $WED$-space $X$  which is not an extremally disconnected space (e.g., $\Bbb R$ with the standard topology). Then, by Lemma \ref{nem}, $C(X)$ is a $W$-ring which is not a $UW$-ring.
    
 (2) A $PW$-ring need not to be a $W$-ring. Consider a $WCF$-space $X$ which is not a $WED$-space (e.g., Example \ref{fati}). Then, by  Theorem \ref{sot} and Lemma \ref{nem}, $C(X)$ is a $PW$-ring which is not a $W$-ring.
 
 (3) A $PW$-ring need not be a $UPW$-ring. Consider a $WCF$-space $X$ which is not an $F$-space (e.g., Example 3.7). Then,  by  Theorem \ref{sot} and Corollary \ref{exa}, $C(X)$ is a $PW$-ring which is not a $UPW$-ring.
 \end{exam}   






\end{document}